\documentclass[submission]{FPSAC2021}

\usepackage[utf8]{inputenc}
\usepackage[english]{babel}

\usepackage{amssymb}
\usepackage{amsmath}
\usepackage{amscd}
\usepackage{units}
\usepackage{textcomp}
\usepackage{lscape}
\usepackage{geometry}
\usepackage{mathrsfs}

\usepackage{paralist}
\usepackage{hyperref}
\usepackage{lineno}

\usepackage{tikz}
\usetikzlibrary{shapes,shadows,arrows,positioning,graphs,patterns}
\usetikzlibrary{calc,decorations.pathmorphing,intersections}
\usetikzlibrary{arrows.meta}

\newtheorem{theorem}{Theorem}[section]
\newtheorem{lemma}[theorem]{Lemma}
\newtheorem{corollary}[theorem]{Corollary}
\newtheorem{conjecture}[theorem]{Conjecture}
\newtheorem{proposition}[theorem]{Proposition}
\theoremstyle{definition} 
\newtheorem{definition}[theorem]{Definition}
\newtheorem{remark}[theorem]{Remark}
\newtheorem{example}[theorem]{Example}

\usepackage{ifthen}
\newboolean{draft}
\setboolean{draft}{false}
\ifdraft
\usepackage{soul}
\usepackage[colorinlistoftodos]{todonotes}
\presetkeys{todonotes}{size=\small,color=green!30,inline,color=red!30}{}
\else
\newcommand{\todo}[2][]{}
\newcommand{\texthl}[1]{}
\fi
\newboolean{longversion}
\setboolean{longversion}{false}

\title{A Bijection Between Weighted Dyck Paths and 1234-avoiding Up-Down Permutations}

\author{Justine Falque}

\address{Univ. Gustave Eiffel, Laboratoire d'Informatique Gaspard Monge,
Marne-la-Vallée; 
France}

\received{\today}

\abstract{Three-dimensional Catalan numbers are a variant of the classical
(bidimensional) Catalan numbers, that count, among other interesting objects,
the standard Young tableaux of shape $(n,n,n)$. In this paper, we present
a structural bijection between two three-dimensional Catalan objects: 
1234-avoiding up-down permutations, and a class of weighted Dyck paths.}

\keywords{Bijective combinatorics, three-dimensional Catalan numbers,
  up-down permutations, pattern avoidance, weighted Dyck paths, Young tableaux, prographs}

\usepackage[backend=bibtex]{biblatex}
\begin{document}

\maketitle

\section{Introduction}

Among a vast amount of combinatorial classes of objects, the
famous Catalan numbers enumerate the standard Young tableaux of shape
$(n,n)$. Counting the standard tableaux of shape $(n,n,n)$ is a sequence
known as the three-dimensional Catalan numbers 
\href{https://oeis.org/A005789}{A005789}, whose first entries are
$1, 1, 5, 42, 462, 6006, 87516, 1385670, \ldots$ .
Many other combinatorial objects are enumerated by this sequence, from certain walks
in the quarter plane~\cite{bousquet_mishna_walks}, to product-coproduct prographs~\cite{Borie17}, to 1234-avoiding up-down permutations.

This last case, which this paper dwells on, was proven by Lewis, who
provides in~\cite{Lewis1,Lewis2} two bijections between this class
and standard Young tableaux 
of shape $(n,n,n)$. 
As observed by Borie in~\cite{Borie17} however, these bijections do not 
highlight any obvious similarities of combinatorial nature.

In this same article, Borie proves that product-coproduct prographs are 
three-dimen\-sional Catalan objects, by giving a bijection with standard
Young tableaux of shape $(n,n,n)$; on the other hand, and with another bijection
involving his prographs, he highlights a certain class of weighted Dyck paths as a
new three-dimensional Catalan family.
About this new family, he makes the freely rephrased following conjecture,
which is the starting point of this paper.

\begin{conjecture}[Conjecture of Borie~\cite{Borie17}]
\label{conj1234}
  There exists a combinatorial bijection between up-down permutations of 
  size $2n$ avoiding $1234$ and a certain class of weighted Dyck paths.
\end{conjecture}

Borie additionally assumes that the positions of the steps $(1,1)$ in the
paths should correspond to the bottom elements in the permutation,
and came up with a partial bijection, in the particular case where these are
exactly the elements $1,2,\ldots,n$. He relied on the observation that
the product-coproduct prographs, in this case, were essentially pairs of binary trees,
and he used a bijection to $123$-avoiding permutations on each one
in such a way that the two permutations could respectively become the 
bottom elements and
top elements of a $1234$-avoiding up-down permutation.

In this extended abstract, we present a general bijection 
from these weighted Dyck paths to $1234$-avoiding up-down permutations.
This bijection extends Borie's partial bijection to the whole combinatorial
classes and preserves some structural properties and statistics.

\section{The combinatorial objects}
This section presents, in two separate subsections, both
classes of combinatorial objects dealt with in this paper.
In each case, we recall the definition, provide examples, and
then describe the Schützenberger involution and a natural product
on the objects.

\subsection{Up-down permutations of $2n$ avoiding $1234$}~\label{section.up_down}

An up-down permutation of $2n$ avoiding $1234$ is a
permutation of size $2n$ whose descents set is $\{2, 4, 6, \dots \}$,
with no increasing subsequence of length $4$. We denote by $A_{2n}(1234)$
the set of all these permutations.

For instance, here are the $42$ up-down permutations of size $6$ avoiding $1234$:
\begin{equation}
  A_{6}(1234) = \left\{
  \begin{array}{c}
    143625, 153624, 154623, 163524, 164523, 241635, 243615,  \\
    251436, 251634, 253614, 254613, 261435, 261534, 263514,  \\
    264513, 341625, 342615, 351426, 351624, 352416, 352614,  \\
    354612, 361425, 361524, 362415, 362514, 364512, 451326,  \\
    451623, 452316, 452613, 453612, 461325, 461523, 462315,  \\
    462513, 463512, 561324, 561423, 562314, 562413, 563412
  \end{array}
  \right\}.
\end{equation}

We shall use the following convenient notations: 
for any permutation $\sigma$ of $A_{2n}(1234)$ seen as a word, we denote by
$Bot(\sigma) = \sigma_2 \sigma_4 \cdots \sigma_{2n}$ the subword that consists of 
the letters in even positions --- that we may call bottom elements rather than
valleys, in order to avoid confusion with the valleys of Dyck paths; 
and by $Top(\sigma) = \sigma_1 \sigma_3 \cdots \sigma_{2n-1}$ the subword of the odd-position letters --- that we may call top elements.
Of course, $\sigma$ is determined by $Bot(\sigma)$ and $Top(\sigma)$, for
instance:

\hspace{2.5cm}
\begin{tikzpicture}[scale=.6]
\draw (-5,0.5) node{$\sigma = 364512$};
\draw (0,0) node{3};
\draw (1.5,0) node{4};
\draw (3,0) node{1};
\draw (0.75,1) node{6};
\draw (2.25,1) node{5};
\draw (3.75,1) node{2};

\draw[<-,>=stealth] (6,0) -- (7,0);
\draw[<-,>=stealth] (6,1) -- (7,1);
\draw (9,0) node{$Bot(\sigma)~.$};
\draw (9,1) node{$Top(\sigma)$};
\end{tikzpicture}


The set of $1234$-avoiding up-down permutations is endowed with a product:
\begin{lemma}[\cite{Borie17}]
  The set $\bigcup_{n \in \mathbb{N}} A_{2n}(1234)$ is closed under the 
  shifted concatenation product $\bullet$ on permutations
  defined by $\sigma \bullet \tau = (\operatorname{shift}_{length(\tau)}(\sigma)) 
  \cdot \tau$.
\end{lemma}

For instance, we have ${\color{red}12} \bullet 1423 = {\color{red}56}1423$.
\\


Another structure indicator, the classical Schützenberger involution on 
permutations consists in reversing
the alphabet, then reversing the reading direction. 

For example,
we have $S(48271635) = 46382715$. 
As it preserves appearance and avoidance of patterns, $S$ stabilizes the 
set of up-down permutations of $2n$ avoiding $1234$.

We shall use a variant on words:
for any word $\omega$ on the alphabet $\{1, 2, \ldots, 2n\}$,  
$S_{2n}(\omega)$ is obtained by the same process of reversing the alphabet 
and the reading direction. The main difference
is that not all letters need to appear.
For instance, we have $S_{8}(164) = 538$.

\subsection{Weighted Dyck paths}

\begin{definition}~\label{def.wd}
  We denote by $WD_{2n}$ weighted Dyck
  paths of length $2n$ whose weights satisfy the following assertions:
  \begin{compactenum}
    \item all weights are non-negative integers smaller than
      or equal to the lower height;
    \item weights are non-decreasing on successive rises;
    \item weights are non-increasing on successive descents;
    \item On a peak of height $h$, with $d$ and $e$ the weights of its steps, we
      have: $e + d \leqslant h$\,;
    \item On a valley of height $h$, with $d$ and $e$ the weights of its steps, we
      have: $d + e \geqslant h$.
  \end{compactenum}
\end{definition}

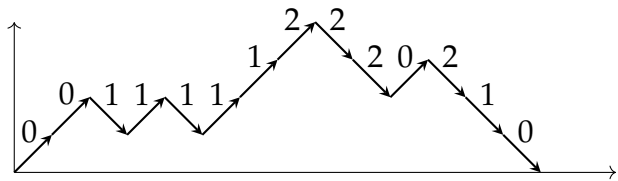
\begin{figure}[h!]
  \centering
  \begin{tikzpicture}[scale=0.5]
    \tikzstyle{p}=[->, thick, >=stealth]
    
    \draw[->] (0,0) -- (16, 0);
    \draw[->] (0,0) -- (0, 4);
    
    \draw[p] (0,0) -- (1,1);  
    \draw[p] (1,1) -- (2,2);  
    \draw[p] (2,2) -- (3,1);  
    \draw[p] (3,1) -- (4,2);  
    \draw[p] (4,2) -- (5,1);  
    \draw[p] (5,1) -- (6,2);  
    \draw[p] (6,2) -- (7,3);  
    \draw[p] (7,3) -- (8,4);  
    \draw[p] (8,4) -- (9,3);  
    \draw[p] (9,3) -- (10,2); 
    \draw[p] (10,2) -- (11,3); 
    \draw[p] (11,3) -- (12,2); 
    \draw[p] (12,2) -- (13,1); 
    \draw[p] (13,1) -- (14,0); 

    \draw (0.4, 0.5) node[above] {$0$};
    \draw (1.4, 1.5) node[above] {$0$};
    \draw (2.6, 1.5) node[above] {$1$};
    \draw (3.4, 1.5) node[above] {$1$};    
    \draw (4.6, 1.5) node[above] {$1$};
    \draw (5.4, 1.5) node[above] {$1$};
    \draw (6.4, 2.5) node[above] {$1$};
    \draw (7.4, 3.5) node[above] {$2$};
    \draw (8.6, 3.5) node[above] {$2$};
    \draw (9.6, 2.5) node[above] {$2$};
    \draw (10.4, 2.5) node[above] {$0$};
    \draw (11.6, 2.5) node[above] {$2$};
    \draw (12.6, 1.5) node[above] {$1$};
    \draw (13.6, 0.5) node[above] {$0$};
  \end{tikzpicture}
  \caption{Example of a weighted Dyck path in $DW_{14}$.}~\label{pathsweights}
  \label{figure.example_wd}
\end{figure}

There is a natural concatenation product on these objects, as well as
a natural notion of ``Schützenberger involution''
: the reflection according to a vertical axis (see 
Figure~\ref{fig.concat_schutz}).

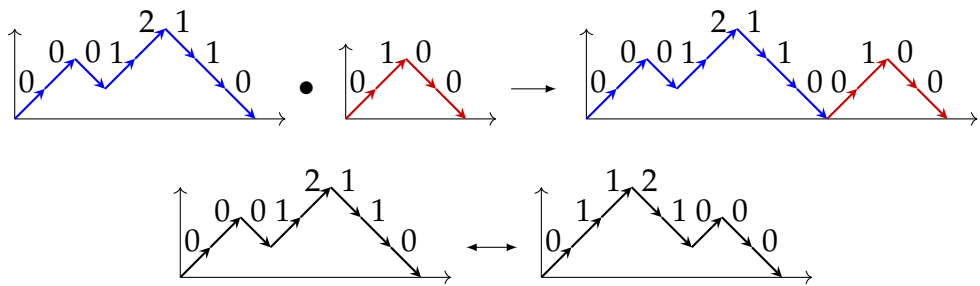
\begin{figure}
  \centering
  \begin{tikzpicture}[scale=0.4]
    \tikzstyle{p}=[->, thick, >=stealth]
    
    \draw[->] (0,0) -- (9, 0);
    \draw[->] (0,0) -- (0, 3);
    
    \draw[p, blue] (0,0) -- (1,1);  
    \draw[p, blue] (1,1) -- (2,2);  
    \draw[p, blue] (2,2) -- (3,1);  
    \draw[p, blue] (3,1) -- (4,2);  
    \draw[p, blue] (4,2) -- (5,3);  
    \draw[p, blue] (5,3) -- (6,2);  
    \draw[p, blue] (6,2) -- (7,1);  
    \draw[p, blue] (7,1) -- (8,0);  
    \draw (0.4, 0.5) node[above] {$0$};
    \draw (1.4, 1.5) node[above] {$0$};
    \draw (2.6, 1.5) node[above] {$0$};
    \draw (3.4, 1.5) node[above] {$1$};
    \draw (4.4, 2.5) node[above] {$2$};
    \draw (5.6, 2.5) node[above] {$1$};    
    \draw (6.6, 1.5) node[above] {$1$};
    \draw (7.6, 0.5) node[above] {$0$};
    
    \draw (9.7,1) node{$\bullet$};
    
    \draw[->] (11,0) -- (16, 0);
    \draw[->] (11,0) -- (11, 2.5);
    
    \draw[p, red!80!black] (11,0) -- (12,1);  
    \draw[p, red!80!black] (12,1) -- (13,2);  
    \draw[p, red!80!black] (13,2) -- (14,1);  
    \draw[p, red!80!black] (14,1) -- (15,0);  
    \draw (11.4, 0.5) node[above] {$0$};
    \draw (12.4, 1.5) node[above] {$1$};
    \draw (13.6, 1.5) node[above] {$0$};
    \draw (14.6, 0.5) node[above] {$0$};
        
    \draw[->, >=latex] (16.5,1) -- (18, 1);
    
    \draw[->] (19,0) -- (32, 0);
    \draw[->] (19,0) -- (19, 3);
    
    \draw[p, blue] (19,0) -- (20,1);  
    \draw[p, blue] (20,1) -- (21,2);  
    \draw[p, blue] (21,2) -- (22,1);  
    \draw[p, blue] (22,1) -- (23,2);  
    \draw[p, blue] (23,2) -- (24,3);  
    \draw[p, blue] (24,3) -- (25,2);  
    \draw[p, blue] (25,2) -- (26,1);  
    \draw[p, blue] (26,1) -- (27,0);  
    \draw (19.4, 0.5) node[above] {$0$};
    \draw (20.4, 1.5) node[above] {$0$};
    \draw (21.6, 1.5) node[above] {$0$};
    \draw (22.4, 1.5) node[above] {$1$};
    \draw (23.4, 2.5) node[above] {$2$};
    \draw (24.6, 2.5) node[above] {$1$};    
    \draw (25.6, 1.5) node[above] {$1$};
    \draw (26.6, 0.5) node[above] {$0$};
    \draw[p, red!80!black] (27,0) -- (28,1);  
    \draw[p, red!80!black] (28,1) -- (29,2);  
    \draw[p, red!80!black] (29,2) -- (30,1);  
    \draw[p, red!80!black] (30,1) -- (31,0);  
    \draw (27.4, 0.5) node[above] {$0$};
    \draw (28.4, 1.5) node[above] {$1$};
    \draw (29.6, 1.5) node[above] {$0$};
    \draw (30.6, 0.5) node[above] {$0$};
  \end{tikzpicture}
  
  \vspace{.4cm}
  
  \begin{tikzpicture}[scale=0.4]
    \tikzstyle{p}=[->, thick, >=stealth]
    
    \draw[->] (0,0) -- (9, 0);
    \draw[->] (0,0) -- (0, 3);
    
    \draw[p] (0,0) -- (1,1);  
    \draw[p] (1,1) -- (2,2);  
    \draw[p] (2,2) -- (3,1);  
    \draw[p] (3,1) -- (4,2);  
    \draw[p] (4,2) -- (5,3);  
    \draw[p] (5,3) -- (6,2);  
    \draw[p] (6,2) -- (7,1);  
    \draw[p] (7,1) -- (8,0);  
    \draw (0.4, 0.5) node[above] {$0$};
    \draw (1.4, 1.5) node[above] {$0$};
    \draw (2.6, 1.5) node[above] {$0$};
    \draw (3.4, 1.5) node[above] {$1$};
    \draw (4.4, 2.5) node[above] {$2$};
    \draw (5.6, 2.5) node[above] {$1$};    
    \draw (6.6, 1.5) node[above] {$1$};
    \draw (7.6, 0.5) node[above] {$0$};
    
    \draw[<->,>=latex] (9.5,1) -- (11.2,1);
    
    \draw[->] (12,0) -- (21, 0);
    \draw[->] (12,0) -- (12, 3);
    
    \tikzstyle{p}=[<-, thick, >=stealth]
    \draw[p] (20,0) -- (19,1);  
    \draw[p] (19,1) -- (18,2);  
    \draw[p] (18,2) -- (17,1);  
    \draw[p] (17,1) -- (16,2);  
    \draw[p] (16,2) -- (15,3);  
    \draw[p] (15,3) -- (14,2);  
    \draw[p] (14,2) -- (13,1);  
    \draw[p] (13,1) -- (12,0);  
    \draw (19.6, 0.5) node[above] {$0$};
    \draw (18.6, 1.5) node[above] {$0$};
    \draw (17.4, 1.5) node[above] {$0$};
    \draw (16.6, 1.5) node[above] {$1$}; 
    \draw (15.6, 2.5) node[above] {$2$};
    \draw (14.4, 2.5) node[above] {$1$};    
    \draw (13.4, 1.5) node[above] {$1$};
    \draw (12.4, 0.5) node[above] {$0$};
  \end{tikzpicture}
  \caption{Concatenation product and Schützenberger involution in $WD_{2n}$.}~\label{fig.concat_schutz}
\end{figure}


\section{A statistics-preserving bijection}

This section presents the results of our work, specifically a solution
to the following conjecture.

\begin{conjecture}[Conjecture of Borie~\cite{Borie17}]
\label{conj_nicolas}
  There exists a bijection between up-down permutations of size $2n$
  avoiding $1234$ and weighted Dyck paths of $WD_{2n}$ that has the
  following properties: compatibility
  with the concatenation product and the Schützenberger involution;
  correspondence between the positions of steps $(1,1)$ in the
  Dyck path and the bottom elements of the permutation.
\end{conjecture}

Let $wd$ be an irreducible weighted Dyck path (that is, a weighted
Dyck path with no intermediate return to 0) 
of length $2n$ and let $wd(u)$ denote the weight 
associated with the step in position $u \in \{1,\ldots,2n\}$.
Let us define a map $\beta'$ on 
irreducible weighted Dyck paths. The image permutation is computed using
an algorithm that inserts the \emph{positions} of the steps one at a time;
moreover, the elements of the bottom word are the positions of the steps $(1,1)$
in the Dyck path.

\begin{definition}
\label{definition.bij}
Let $\beta'(wd)$ be the permutation $\sigma$ whose bottom word is
obtained by an insertion algorithm $ins$ defined below;
and the top word is obtained by applying the
Schüztenberger involution to $wd$, then the algorithm $ins$, and then
the (shifted) Schützenberger involution again (so that in the end the 
elements in the top word are the positions of steps $(1,-1)$ in $wd$).
Formally, $\sigma$ is defined by:
\vspace{-.2cm}
\begin{align*}
Bot(\sigma) &= ins(wd)\\
Top(\sigma) &= S_{2n}(ins(S(wd)))
\end{align*}

\vspace{-.2cm}
where $S$ and $S_{2n}$ are the 
Schützenberger involution and its variant on words,
respectively, and $ins$ is the function defined by the following
algorithm.
\begin{enumerate}
\item Split $wd$ into the first, left-hand
half of slopes $L$ and the second, right-hand half of slopes $R$
(see Figures~\ref{figure.locations_comp}\,\&\,\ref{figure.slopes_up} for examples).
\item 
Start with $\tau = \varepsilon$ the empty permutation. The weighted
Dyck path $wd$ is scanned from the left, and elements are inserted
in $\tau$ as we go.
\item
For every upward slope $U$ in $wd$, starting from the left

\vspace{-.3cm}
\begin{enumerate}
  \item Let $shift$ be the total number of steps $(1,-1)$ that are to the left 
  of the slope in $wd$.
  \item For every step $(1,1)$ of the current upward slope, starting from the left, 
  denote by $u$ its position and 
  %
  %
  proceed 
  to insert $u$ in $\tau$ as follows:
  \begin{itemize}
    \item 
    In case the slope $U$ is in $L$: if $wd(u)$ is \emph{minimal}, 
    insert $u$ to the left of $\tau$ (we say that $u$ \emph{jumps},
    or that $wd(u)$ is a jumping weight); 
    otherwise, insert $u$ in $\tau$ at a distance 
    $wd(u) + shift -1$ from the right-hand end of $\tau$.
    \item In case the slope $U$ is in $R$: 
    consider instead the \emph{maximality} of $wd(u)$ 
    to decide if $u$ jumps, and insert at a distance 
    $wd(u) + shift$ from the right-hand end otherwise.

  \end{itemize}

    The minimality (resp. maximality) of the value is decided by considering
    either the increment condition on the slope or the 
    valley (resp. peak) condition of Definition~\ref{def.wd}: if the
    corresponding bound is achieved, then the tested value is indeed extremal.
\end{enumerate}
\end{enumerate}
\end{definition}

\begin{figure}
\begin{center}
\begin{tikzpicture}
\draw[thick] (-1,0) -- (0,0); 
\draw[dashed] (0,0) -- (1,0); 
\draw[thick] (1,0) -- (4,0);
\draw[dashed] (4,0) -- (5,0);
\draw[thick] (5, 0) -- (8.6, 0);
\draw[dashed] (8.6,0) -- (9.4,0);
\draw[thick] (9.4, 0) -- (11, 0); 
\draw[thick, red, ->, >=stealth] (2,.3) to[bend right] (-.9,0.3);
\draw[thick, violet, <->, >=stealth] (7.5,.6) -- (11.2,.6);
\draw (9,.8) node{\color{violet}$shift$};
\draw[thick, violet!80] (7.5,.6) -- (7.5,-.6);
\draw[thick, violet!80] (11.2,.6) -- (11.2,-.6);
\fill [pattern=north east lines, pattern color=violet!30] (7.5,.6) rectangle (11.2,-.6);
\draw[->, >=stealth, blue] (11, -.4) -- ++(0,.3);
\draw (11.02, -.8) node {\color{blue}$0$};
\draw[->, >=stealth, blue] (10, -.4) -- ++(0,.3);
\draw (10.02, -.8) node {\color{blue}$1$};
\draw[->, >=stealth, blue] (9, -.4) -- ++(0,.3);
\draw[->, >=stealth, blue] (8, -.4) -- ++(0,.3);
\draw[->, >=stealth, blue] (7, -.4) -- ++(0,.3);
\draw (7, -.8) node {\small \color{blue}$shift$};
\draw[->, >=stealth, blue] (6, -.4) -- ++(0,.3);
\draw[->, >=stealth, blue] (5, -.4) -- ++(0,.3);
\draw[->, >=stealth, blue] (4, -.4) -- ++(0,.3);
\draw[->, >=stealth, blue] (3, -.4) -- ++(0,.3);
\draw[->, >=stealth, red] (-1, -.4) -- ++(0,.3);
\draw (4.5, -.8) node {\color{blue}$\ldots$};
\draw[->, >=stealth, red] (2, -.4) -- ++(0,.3);
\draw (2, -.8) node {\small \color{red}$max$};
\draw (-.5, -.8) node {\small \color{red}jump landing};
\draw (-4.4, -.2) node {\small Insertion positions};
\draw (-4.4, -.8) node {\small from the right:};
\draw (-4, 1) node {Bottom word};

\draw (10.5, 0) node[scale=.6,circle, fill=black]{};
\draw (9.5, 0) node[scale=.6,circle, fill=black]{};
\draw (8.5, 0) node[scale=.6,circle, fill=black]{};
\draw (7.5, 0) node[scale=.6,circle, fill=black]{};
\draw (6.5, 0) node[scale=.6,circle, fill=black]{};
\draw (5.5, 0) node[scale=.6,circle, fill=black]{};
\draw (3.5, 0) node[scale=.6,circle, fill=black]{};
\draw (2.5, 0) node[scale=.6,circle, fill=black]{};
\draw (1.5, 0) node[scale=.6,circle, fill=black]{};
\draw (-0.5, 0) node[scale=.6,circle, fill=black]{};
\end{tikzpicture}
\end{center}
\caption{Insertion positions in the building bottom word (dots are word's elements).}
\label{figure.insertion_positions}
\end{figure}
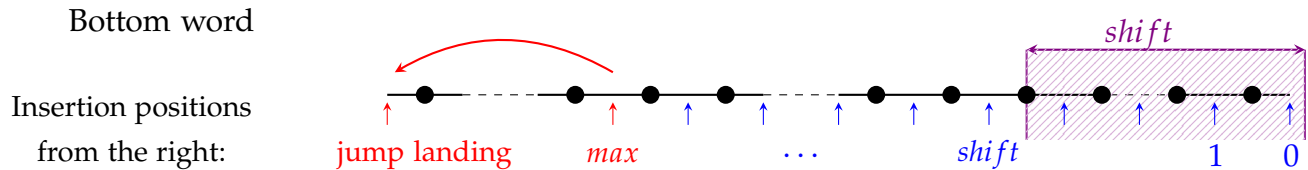

\begin{remark}
Since a bottom (resp. top) element in an up-down permutation is 
automatically followed (resp. preceded) by a larger top (resp. smaller bottom)
element, a top element and a bottom element, in that order, can never form a 
$12$ pattern in $A_{2n}(1234)$. This is the reason for the $shift$ value: 
as $shift$ elements,
among the top elements of the permutation, need to be to the right of the element $u$
that is being inserted as a bottom element, the insertion position of $u$ from the
right needs to be at least $shift$ if we aim at obtaining (as we actually
do) a $1234$-avoiding permutation.
\end{remark}

\begin{remark}
\label{remark.locations_comp_for_jumps}
Keeping all the way through the path the comparison to the 
minimum (resp. maximum) as the unique criterion for jumps may 
seem tempting, but it would make the map $\beta$ 
non-injective, as a single valley
(resp. peak) condition would decide if both of its adjacent steps wield
jumping weights --- this being generally the case with several different pairs of
weights. On the other hand, a choice like comparing to the minimum all weights
of the steps $(1,1)$ and to the maximum all those of the steps down would make this
map incompatible with the Schützenberger involution (refer to
Figure~\ref{figure.locations_comp} for visual support).
\end{remark}

\begin{remark}
Why jump at all? Without diving into the detail, it allows to avoid
completing a $1234$ pattern (recall that is an increasing subsequence of size $4$).
The $ins$ construction has two assets. First, as we explain later, 
elements that do not jump are those forming a $12$ pattern with
their fellow bottom elements; actually, at the time of insertion, all elements to the 
left are the $1$ (the smallest element in the increasing subsequence of length 2) 
of a $12$ pattern, the element not jumping being the $2$.
Second (as a consequence), the insertion position of a no-jump element is the
number of elements strictly following the (current) first ascent in the word; 
transposed to $Top(\sigma)$, this means the weights of steps down enable to 
locate the rightmost ascent. For bottom elements, jumping allows
to escape the threat of this ascent as a potential $34$ in the $1234$.
Since the leftmost insertion position that could be computed, if it was not
for jumps, is the problematic one, one could have thought of just making 
weights achieving this bound jumping weights... but then, not all elements
would be able to jump ---
which brings back to Remark~\ref{remark.locations_comp_for_jumps}.
\end{remark}

\begin{remark}
The minus $1$ in the computation of the insertion position as of the
left-hand half $L$ compensates the fact that elements
jump if their weight is minimal, so the non-jumping weights have values
between $1$ and $\ell$ (the current lower height),
whereas the valid insertion positions are $0$ to $\ell-1$.
\end{remark}

\begin{figure}
\label{figure.min_max}
  \centering
  \begin{tikzpicture}[scale=0.5]
    \tikzstyle{p}=[->, thick, >=stealth]
    
    \draw[->] (0,0) -- (18, 0);
    \draw[->] (0,0) -- (0, 5);
    
    \draw[p, ultra thick] (0,0) -- (1,1);  
    \draw[p, gray] (1,1) -- (2,2);  
    \draw[p, gray] (2,2) -- (3,3);  
    \draw (2.5,4.5) node{\color{blue}$L$} ;  
    \draw[p, ultra thick] (3,3) -- (4,2);  
    \draw[p, gray] (4,2) -- (5,1);  
    \draw[p, ultra thick] (5,1) -- (6,2);  
    \draw[p, gray] (6,2) -- (7,3);  
    \draw[p, gray] (7,3) -- (8,4);  
    \draw[gray!60, dashed] (8,0) -- (8,5);  
    \draw[p, gray] (8,4) -- (9,3);  
    \draw[p, ultra thick] (9,3) -- (10,2); 
    \draw[p, gray] (10,2) -- (11,3); 
    \draw[p, ultra thick] (11,3) -- (12,4); 
    \draw[p, gray] (12,4) -- (13,3); 
    \draw (15,4.5) node{\color{orange}$R$} ;  
    \draw[p, gray] (13,3) -- (14,2); 
    \draw[p, gray] (14,2) -- (15,1); 
    \draw[p, ultra thick] (15,1) -- (16,0); 
  \end{tikzpicture}
  \caption{Locations where the jump is decided using the peak or valley condition.}
  \label{figure.locations_comp}
\end{figure}
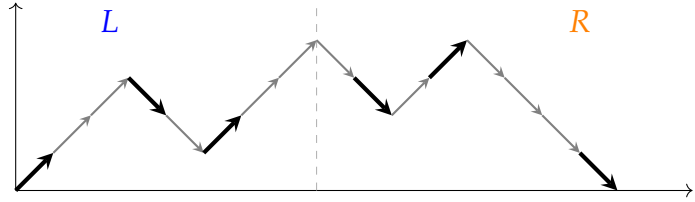

In the sequel, whilst we will be applying the algorithm to weighted Dyck 
paths $wd$ in order to realize $\sigma = \beta'(wd)$, we will be refering to 
the transitional word that will become $Bot(\sigma)$ as the \emph{bottom word}, 
and to the word that will become $Top(\sigma)$ as the \emph{top word}.

\begin{example}
Take again the example in Figure~\ref{figure.example_wd},
of an element of $WD_{14}$.
Here, the set $L$ consists of the upward slopes to the left of the valley in 
position $5$ (in dotted 
blue on Figure~\ref{figure.slopes_up}), and $R$ consists of the upward slopes to 
the right (in dashed orange).
The big dots on the figure mark the places where the weight must be compared
to its minimal/maximal possible value in order to decide if the element
jumps.

\begin{figure}
  \centering
  \begin{tikzpicture}[scale=0.6]
    \tikzstyle{p}=[->, thick, >=stealth]
    
    \draw[->] (0,0) -- (16, 0);
    \draw[->] (0,0) -- (0, 5);
    
    \draw[p, dotted, thick, blue] (0,0) -- (1,1);  
    \draw[p, dotted, thick, blue] (1,1) -- (2,2);  
    \draw (.1,.1) node[scale=.6,circle, fill=blue]{} ;  
    \draw[p] (2,2) -- (3,1);  
    \draw[p, dotted, thick, blue] (3,1) -- (4,2);  
    \draw (3.1,1.1) node[scale=.6,circle, fill=blue]{} ;  
    \draw[p] (4,2) -- (5,1);  
    \draw[gray!80, dashed] (5,4.9) -- (5,0);
    \draw[p, dashed, orange] (5,1) -- (6,2);  
    \draw[p, dashed, orange] (6,2) -- (7,3);  
    \draw[p, dashed, orange] (7,3) -- (8,4);  
    \draw (7.9,3.9) node[scale=.6,circle, fill=orange]{} ;  
    \draw[p] (8,4) -- (9,3);  
    \draw[p] (9,3) -- (10,2); 
    \draw[p, dashed, orange] (10,2) -- (11,3); 
    \draw (10.9,2.9) node[scale=.6,circle, fill=orange]{} ;  
    \draw[p] (11,3) -- (12,2); 
    \draw[p] (12,2) -- (13,1); 
    \draw[p] (13,1) -- (14,0); 

    \draw (0.4, 0.5) node[above] {$0$};
    \draw (1.4, 1.5) node[above] {$0$};
    \draw (2.6, 1.5) node[above] {$1$};
    \draw (3.4, 1.5) node[above] {$1$};    
    \draw (4.6, 1.5) node[above] {$1$};
    \draw (5.4, 1.5) node[above] {$1$};
    \draw (6.4, 2.5) node[above] {$1$};
    \draw (7.4, 3.5) node[above] {$2$};
    \draw (8.6, 3.5) node[above] {$2$};
    \draw (9.6, 2.5) node[above] {$2$};
    \draw (10.4, 2.5) node[above] {$0$};
    \draw (11.6, 2.5) node[above] {$2$};
    \draw (12.6, 1.5) node[above] {$1$};
    \draw (13.6, 0.5) node[above] {$0$};

    \draw (3, 4) node{\color{blue}$L$};
    \draw (12.5, 4) node{\color{orange!80!red}$R$};
    \draw (0.5, -0.8) node{\color{blue}$1$};
    \draw (1.5, -0.8) node{\color{blue}$2$};
    \draw (2.5, -0.8) node{\color{gray!50}$3$};
    \draw (3.5, -0.8) node{\color{blue}$4$};
    \draw (4.5, -0.8) node{\color{gray!50}$5$};
    \draw (5.5, -0.8) node{\color{orange!80!red}$6$};
    \draw (6.5, -0.8) node{\color{orange!80!red}$7$};
    \draw (7.5, -0.8) node{\color{orange!80!red}$8$};
    \draw (8.5, -0.8) node{\color{gray!50}$9$};
    \draw (9.5, -0.8) node{\color{gray!50}$10$};
    \draw (10.5, -0.8) node{\color{orange!80!red}$11$};
    \draw (11.5, -0.8) node{\color{gray!50}$12$};
    \draw (12.5, -0.8) node{\color{gray!50}$13$};
    \draw (13.5, -0.8) node{\color{gray!50}$14$};
    \draw (15.8, -0.8) node{\color{black}$u$};
  \end{tikzpicture}
  \caption{Left-hand and right-hand halves of the upward slopes ($L$ and $R$, respectively).}~\label{figure.slopes_up}
\end{figure}
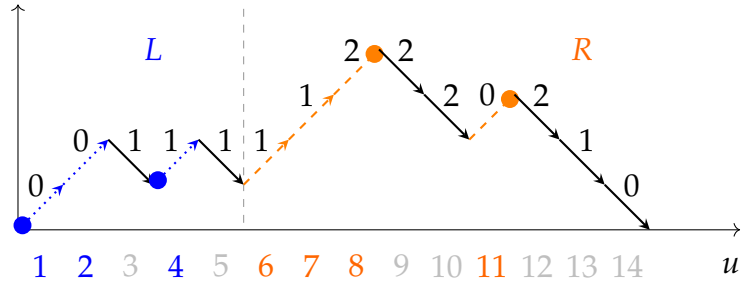

\begin{compactenum}
\item First to be handled is the leftmost upward slope. Both elements 
jump to the left: $2~1$\,.
\vspace{.3cm}

\item We overrun the first slope down, which is one step long,
so we have $shift = 1$ as we handle the $4$. Since this upward slope
still belongs to the first half $L$, we compare the weight $1$ to its
minimum value. In the absence of a previous weight on the slope
to compare it to, it
is determined by the valley condition: here, $1+1$ is strictly greater
than the valley height $1$, so the weight is not minimal and $4$ does
not jump. Its insertion position from the right is $1+shift-1 = 1$, so we obtain
$2~{\color{red!80!black}4}~1$\,.
\vspace{.3cm}

\item We overrun the second slope down and add its length to $shift$, which is
now $2$. We leave the left-hand half $L$ and enter the right-hand half $R$,
so we need to check if the next weight is maximal to decide if the elements jump.
The value $1$ is indeed maximal, both because of its height and because the next
element is $1$ too, so $6$ jumps.
However, $7$ does not, since $1$ is less than both its lower height and
the following weight $3$.
In $R$, we do not subtract $1$ to the insertion position
anymore; we get $1+shift = 3$ for $7$, and: ${\color{red!80!black}6~7}~2~4~1$\,.
Finally, $8$ does jump because $3$ is maximal with respect to the peak condition:
$2+2 = 4 = h_{peak}$. We have now: ${\color{red!80!black}8}~6~7~2~4~1$\,.
\vspace{.3cm}

\item We add the length of the next slope down to $shift$, which brings it to $4$,
and test if $0$ is maximal. It is obviously smaller than its height, and since 
it is the last element of the slope we need to test the peak condition: 
$0+2$ is less than $3$, so $11$ does not jump and is inserted at the position
$0+shift = 4$ from the right, which hands:
\vspace{-.3cm}

\[
Bot(\sigma) = 8~6~{\color{red!80!black}11}~7~2~4~1\,.
\]

\vspace{-.1cm}
\item The subsequence $Top(\sigma)$ is obtained by applying the Schützenberger
involution to the path, basically reversing left and right, then executing the
algorithm, and finally applying the (shifted) Schützenberger involution
so as to be back with the right element values. Note that it boils down to 
using the same algorithm on the slopes down as on the upward slopes, except they are 
scanned from right
to left (with the rules related to $L$ and $R$ swapped), the insertion position
computed is from the left, and elements that jump go all the way to the right.
In the end we get $Top(\sigma) = 13~12~14~10~9~5~3$ and $\sigma = 8~13~6~12~11~14~7~10~2~9~4~5~1~3$\,.
\end{compactenum}
\end{example}

\begin{definition}
Let $wd$ be a weighted Dyck path in $WD_{2n}$. We define $\beta(wd)$ as
the map obtained by applying $\beta'$ to the irreducible factors of $wd$
and then taking the shifted concatenation product of the image permutations.
\end{definition}

\begin{remark}
By construction, the map $\beta$ is compatible with the concatenation
product as well as the Schützenberger involution of both $WD_{2n}$ and $A_{2n}(1234)$.
\end{remark}

\begin{theorem}
The map $\beta$ is a bijection between $DW_{2n}$ and $A_{2n}(1234)$.
\end{theorem}


\section{Proof of the bijection}
\subsection{Proof of injectivity}

We have the following lemma.

\begin{lemma}
\label{lemma.different_positions}
Let $wd$ be an irreducible element of $WD_{2n}$. Suppose we are applying
$\beta'$ to $wd$ using the algorithm of Definition~\ref{definition.bij}, and call
$u$ the next element to be
inserted in the bottom word. Assume further that
$u$ does not jump. Then the possible values of
the weight $wd(u)$ all hand a different, valid insertion position.
\end{lemma}

\iflongversion
\begin{proof}
We prove it for the bottom word, and under this hypothesis that
$u$ belongs to a upward slope that is in $R$, as the proof is
essentially the same for the other cases.
The current length of the word is the number $\ell$ of steps $(1,1)$
already handled, that is, the number of elements already inserted.
As $u$ does not jump, it will be inserted
at position $wd(u) + shift$, where $shift$ is the number of steps
down that have been overrun at this stage of the algorithm. 
The range that should be granted to
$wd(u)$ is thus between $0$ and the length of the allowed insertion zone,
that is $\ell - shift$ (or actually $\ell -shift -1$ since we exclude jump
cases); this is exactly the lower height of the step 
corresponding to $u$, so this condition is ensured by the definition of 
$WD_{2n}$ (and the fact that $u$ does not jump).
\else
\begin{proof}[Sketch of proof]
By considering the range of possible insertion positions in this case.
\fi
\end{proof}

\begin{proposition}
The map $\beta$ is injective.
\end{proposition}

\begin{proof}[Sketch of proof]
By definition, two weighted Dyck paths with different underlying
Dyck paths have different images, since the steps $(1,1)$ correspond exactly
to the elements in even positions in the image permutation.

Let us consider $wd$ and $wd'$ two weighted Dyck paths from $WD_{2n}$
with the same underlying Dyck path and the same image,
and assume that they are different.
Consider the position of the first difference $a$ in the weight values, 
starting from the left: we have $wd(a) \neq wd'(a)$ and
$wd(e) = wd'(e)$ for all $e < a$.

Using Lemma~\ref{lemma.different_positions}, for the element $a$ 
to be inserted in the same 
position in the image permutations, regardless of the difference of
values, it needs to jump in both cases. 

Since $wd(a-1)$ is equal to $wd'(a-1)$, the slope $a$ needs
to be one whose jumps are decided by considering the next
value to the right (so we are in $R$, and we will stay so all the way until
the right-hand end).

By checking every configuration, one can show that every weight 
from there on to the right
has
only one possible value,
determined by the weight of $a$ and that differs in both paths. 
This is necessary all the way to the last element, which can only 
assume the weight 0 by definition; so this is actually impossible.
\end{proof}


\subsection{The image is a 1234-avoiding up-down permutation}

In this subsection, we rely on the 
following criteria.
\begin{proposition}[\cite{Borie17}]~\label{prop_cond}
  For $n$ a non-negative integer and $\sigma$ a permutation of size
  $2n$, $\sigma$ is an up-down permutation avoiding $1234$
  if and only if the following four conditions are satisfied:
  \begin{compactenum}
    \item the sequence $Top(\sigma)$ avoids $123$;
    \item the sequence $Bot(\sigma)$ avoids $123$;
    \item each value of $Top(\sigma)$ smaller than a bottom element $k$ appears
      to the right of $k$ in $\sigma$;
    \item if a bottom element $k$ has a smaller bottom element to its left, all peak
      values greater than $k$ to its right must be ordered in
      $\sigma$ decreasingly.
  \end{compactenum}
\end{proposition}

\begin{proposition}
Let $wd$ be a weighted Dyck path and $\sigma$ be its image by the
map $\beta$. Then, $\sigma$ is an up-down permutation avoiding $1234$.
\end{proposition}

\begin{proof}[Sketch of proof]
We shall use the above criteria to prove the proposition.

First and foremost, let us prove that $Bot(\sigma)$ avoids $123$ (the 
first item is proven essentially the same way). 

Note that the insertion process on the first upward slope (which
hands a permutation of size the length of the slope) is
actually a bijection between non-decreasing parking functions
(here starting from 0) and $123$-avoiding permutations which
is described in~\cite{Borie17}.

To obtain the subsequence $Bot(\sigma)$, one just needs to consider
the upward slopes, as well as whether their valley/peak steps
need to jump (which is usually determined by looking at the value of the 
adjacent step $(1,-1)$). Following this viewpoint, we define a transformation
of the upward slopes, whose image is a single
upward slope; this slope will be such that applying the function $ins$ to 
it hands a permutation that is the
standardized version of $Bot(\sigma)$. To obtain the weights of the
new upward slope from the initial (upward slopes of the) weighted Dyck path, 
consider each step $(1,1)$, from left to right and determine if it jumps
in the initial Dyck path. If so, give it the same image weight as the 
previously obtained weight; if not, the image weight will be 
the pre-image weight to which one adds a \emph{correction}, which is
the number $shift$
of steps $(1,-1)$ that are on its left, plus $1$ if the pre-image weight 
belongs to the right-hand half of the upward slopes (this in order to 
compensate the minus one that is applied to left-hand slopes when
determining the position of insertion according to $ins$).

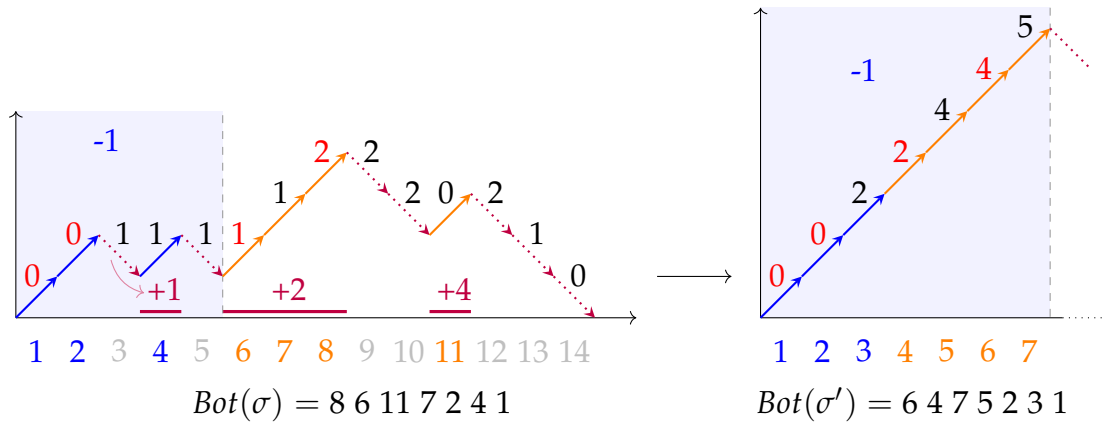
\begin{figure}[h]
  \centering
  \begin{tikzpicture}[scale=0.55]
    \tikzstyle{p}=[->, thick, >=stealth]
    
    \fill[color=blue!5] (0,0) -- (0,5) -- (5,5) -- (5,0) -- (0,0);
    \draw (2.2, 4.3) node{\color{blue} -1};
    \draw[->] (0,0) -- (15, 0);
    \draw[->] (0,0) -- (0, 5);
    \draw (3.6, 0.7) node{\color{purple} +1};
    \draw[very thick, purple] (3,0.15) -- (4, 0.15);
    \draw (6.6, 0.7) node{\color{purple} +2};
    \draw[very thick, purple] (5,0.15) -- (8, 0.15);
    \draw (10.6, 0.7) node{\color{purple} +4};
    \draw[very thick, purple] (10,0.15) -- (11, 0.15);
    \draw[->, purple!50] (2.3,1.4) to[bend right] (3.1,0.6);
    
    \draw[p, blue] (0,0) -- (1,1);  
    \draw[p, blue] (1,1) -- (2,2);  
    \draw[p, dotted, thick, purple] (2,2) -- (3,1);  
    \draw[p, blue] (3,1) -- (4,2);  
    \draw[p, dotted, thick, purple] (4,2) -- (5,1);  
    \draw[gray!80, dashed] (5,4.9) -- (5,0);
    \draw[p, orange] (5,1) -- (6,2);  
    \draw[p, orange] (6,2) -- (7,3);  
    \draw[p, orange] (7,3) -- (8,4);  
    \draw[p, dotted, thick, purple] (8,4) -- (9,3);  
    \draw[p, dotted, thick, purple] (9,3) -- (10,2); 
    \draw[p, orange] (10,2) -- (11,3); 
    \draw[p, dotted, thick, purple] (11,3) -- (12,2); 
    \draw[p, dotted, thick, purple] (12,2) -- (13,1); 
    \draw[p, dotted, thick, purple] (13,1) -- (14,0); 

    \draw (0.4, 0.5) node[above] {\color{red}$0$};
    \draw (1.4, 1.5) node[above] {\color{red}$0$};
    \draw (2.6, 1.5) node[above] {$1$};
    \draw (3.4, 1.5) node[above] {$1$};    
    \draw (4.6, 1.5) node[above] {$1$};
    \draw (5.4, 1.5) node[above] {\color{red}$1$};
    \draw (6.4, 2.5) node[above] {$1$};
    \draw (7.4, 3.5) node[above] {\color{red}$2$};
    \draw (8.6, 3.5) node[above] {$2$};
    \draw (9.6, 2.5) node[above] {$2$};
    \draw (10.4, 2.5) node[above] {$0$};
    \draw (11.6, 2.5) node[above] {$2$};
    \draw (12.6, 1.5) node[above] {$1$};
    \draw (13.6, 0.5) node[above] {$0$};

    \draw (0.5, -0.8) node{\color{blue}$1$};
    \draw (1.5, -0.8) node{\color{blue}$2$};
    \draw (2.5, -0.8) node{\color{gray!50}$3$};
    \draw (3.5, -0.8) node{\color{blue}$4$};
    \draw (4.5, -0.8) node{\color{gray!50}$5$};
    \draw (5.5, -0.8) node{\color{orange}$6$};
    \draw (6.5, -0.8) node{\color{orange}$7$};
    \draw (7.5, -0.8) node{\color{orange}$8$};
    \draw (8.5, -0.8) node{\color{gray!50}$9$};
    \draw (9.5, -0.8) node{\color{gray!50}$10$};
    \draw (10.5, -0.8) node{\color{orange}$11$};
    \draw (11.5, -0.8) node{\color{gray!50}$12$};
    \draw (12.5, -0.8) node{\color{gray!50}$13$};
    \draw (13.5, -0.8) node{\color{gray!50}$14$};

    \draw[->] (15.5,1) -- (17.3,1);
    \begin{scope}[xshift=18cm]
    \fill[color=blue!5] (0,0) -- (0,7.5) -- (7,7.5) -- (7,0) -- (0,0);
    \draw[gray!80, dashed] (7,7.5) -- (7,0);
    \draw (0,0) -- (7.3, 0);
    \draw[->, dotted] (7.3,0) -- (8.5, 0);
    \draw[->] (0,0) -- (0, 7.5);
    \draw (2.5, 6) node{\color{blue} -1};
    
    \draw[p, blue] (0,0) -- (1,1);  
    \draw[p, blue] (1,1) -- (2,2);  
    \draw[p, blue] (2,2) -- (3,3);  
    \draw[p, orange] (3,3) -- (4,4);  
    \draw[p, orange] (4,4) -- (5,5);  
    \draw[p, orange] (5,5) -- (6,6);  
    \draw[p, orange] (6,6) -- (7,7); 
    \draw[purple, dotted, thick] (7,7) -- (8,6); 
    
    \draw (0.4, 0.5) node[above] {\color{red}$0$};
    \draw (1.4, 1.5) node[above] {\color{red}$0$};
    \draw (2.4, 2.5) node[above] {$2$};
    \draw (3.4, 3.5) node[above] {\color{red}$2$};
    \draw (4.4, 4.5) node[above] {$4$};
    \draw (5.4, 5.5) node[above] {\color{red}$4$};
    \draw (6.4, 6.5) node[above] {$5$};
    
    \draw (0.5, -0.8) node{\color{blue}$1$};
    \draw (1.5, -0.8) node{\color{blue}$2$};
    \draw (2.5, -0.8) node{\color{blue}$3$};
    \draw (3.5, -0.8) node{\color{orange}$4$};
    \draw (4.5, -0.8) node{\color{orange}$5$};
    \draw (5.5, -0.8) node{\color{orange}$6$};
    \draw (6.5, -0.8) node{\color{orange}$7$};
    \end{scope}
  \end{tikzpicture}
  
  ~~~~~~~~~~~~~~ $Bot(\sigma) = 8~6~11~7~2~4~1$ \hspace{3cm} 
  $Bot(\sigma') = 6~4~7~5~2~3~1$
  \caption{Transformation to a single slope. Weights are red when there is a jump.}~\label{figure.tranformation}
\end{figure}

\iflongversion
The new weights are weakly increasing. Indeed, it is obvious when the
weight corresponds to a jump (it is equal to the previous one) or for
weights formerly belonging to the same slope (since one adds a constant
value to a slope); consider now a formerly last element of a slope 
$w$ and the formerly first element of the following slope $w'$.
The peak condition demands that $u$ be less than or equal to 
$h_{peak}-d$, where $h_{peak}$ is the height of the peak and $d$ the 
adjacent weight. On the other hand, we have $w' \geqslant h_{val} - d'$,
where $h_{val}$ is the height of the valley of $w'$ and $d'$ the
adjacent weight on this valley. But $h_{val}-d'$ is greater than
$h_{val}-d=h_{peak}-\ell-d$, which is in turn greater than $w-\ell$, where 
$\ell$ is the length of the slope down between the two considered
upward slopes, and also exactly the difference between the \emph{corrections} 
added by the transformation to $w$ and $w'$ (except for maybe an additional
$1$ that works in our favour).
\else
First, use consecutive peak and valley conditions to show that the weights 
of the transformation are non-decreasing.
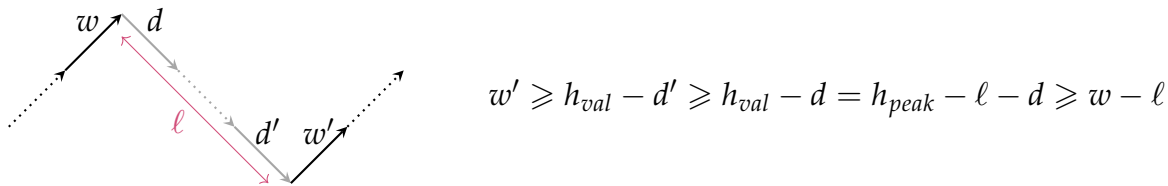
\fi

  \begin{tikzpicture}[scale=0.75]
    \tikzstyle{p}=[->, thick, >=stealth]
    
    \draw[p, dotted, thick] (0,0) -- (1,1);  
    \draw[p] (1,1) -- (2,2);  
    \draw[p, gray!70] (2,2) -- (3,1);  
    \draw[p, gray!70, thick, dotted] (3,1) -- (4,0);  
    \draw[p, gray!70] (4,0) -- (5,-1);  
    \draw[p] (5,-1) -- (6,0);  
    \draw[p, dotted, thick] (6,0) -- (7,1);  
    \draw[<->, purple!70] (2,1.6) -- (4.6,-1);
    
    \draw (1.4, 1.5) node[above] {$w$};
    \draw (2.6, 1.5) node[above] {$d$};
    \draw (4.6, -0.5) node[above] {$d'$};
    \draw (5.5, -0.5) node[above] {$w'$};
    \draw (3, 0.5) node[below] {\color{purple!70}$\ell$};
    \draw (14.5, 0.5) node {$w'\geqslant h_{val}-d' \geqslant h_{val}-d = h_{peak}-\ell-d \geqslant w-\ell$};
  \end{tikzpicture}

In addition, the new weights are still bounded by the height. It is straightforward
for the left-hand half of upward slopes (in blue on the figures), since the weight
and height of the step are increased by the same value. As for the right-hand half
(in orange), to which an additionnal 1 is added,
the only values that could bring trouble are maximal values with respect to the
height and, as part of the right-hand half, that means they jump: they will
thus take the value of the previous weight in the image, so they necessarily
stay below the bound.

The new weights are therefore a non-decreasing parking function (starting from
0), on which the function $ins$ hands a $123$-avoiding permutation, which
ends the proof.\\

We now move on to the third criterion. Refering to
the notations of Figure~\ref{fig.small_peak_right}, we need to show that the distance
between $d$ and $u$ in the image permutation (in grey) is non-negative.
We compute it as the difference between the distances from the right
of $u$ and $d$ (in green and violet, respectively).
We assume that neither $d$ nor $u$ jumps, which is the worst-case scenario
for the distance. More specifically, we may assume without loss of generality
that at least one element from each one of the two slopes does not jump; 
the lowest height 
element of the slope that does not jump is then the one we should examine since
it is the closest, in the permutation, to its counterparts from the other 
slope: replace $d$ (resp. $u$) by this element.

In the algorithm of insertion defined by $\beta$, $u$ is inserted at a
distance from the right $pos(u) = wd(u) + shift(u) ~(-1)$, where $wd(u)$ is the weight
of $u$ in the Dyck path, $shift(u)$ is defined in the algorithm (see also the
legend of Figure~\ref{fig.small_peak_right}), 
and $1$ may or may not be subtracted depending
on the position of the slope in the path.
What the transformation described in the previous item of the proof shows is that
the steps $(1,1)$ that are to the right of $u$ correspond to elements that will be 
inserted to its left in the permutation (this is only true if $u$ does not
jump). Therefore, this position of insertion is also the final distance from
the right of $u$. 
For the same reason (except that the insertion position is computed as a distance
from the left), the final distance from the right of $d$ is
$n-pos(d) = n - wd(d) + shift(d) ~(+1)$.
Note also, by exhaustion of cases, that at most one of the two
insertion positions requires a minus $1$.

Now the difference of positions between $d$ and the element following $u$ in 
the permutation is:

\vspace{-1cm}
\begin{align*}
dist(u,d) &= pos(u) +pos(d)-n ~~(-1)\\
&= wd(u) + wd(d) +shift(u) + shift(d)-n ~~(-1)\\
&= wd(u)+wd(d)+shift(u)-lenBot(u) ~~(-1)\\
&= wd(u)+wd(d)-h_{val} ~~(-1)~\\
\end{align*}
\vspace{-.6cm}

\vspace{-.7cm}
where $lenBot(u)$ is the
length of the bottom word right before the insertion of the slope of $u$.
The condition on the valley is that we have $wd(u)+wd(d)\geqslant h_{val}$,
but since neither $u$ nor $d$ jumps, this inequality is strict, so that
the distance is non-negative, and $d$ is before $u$ in the permutation, as needed.

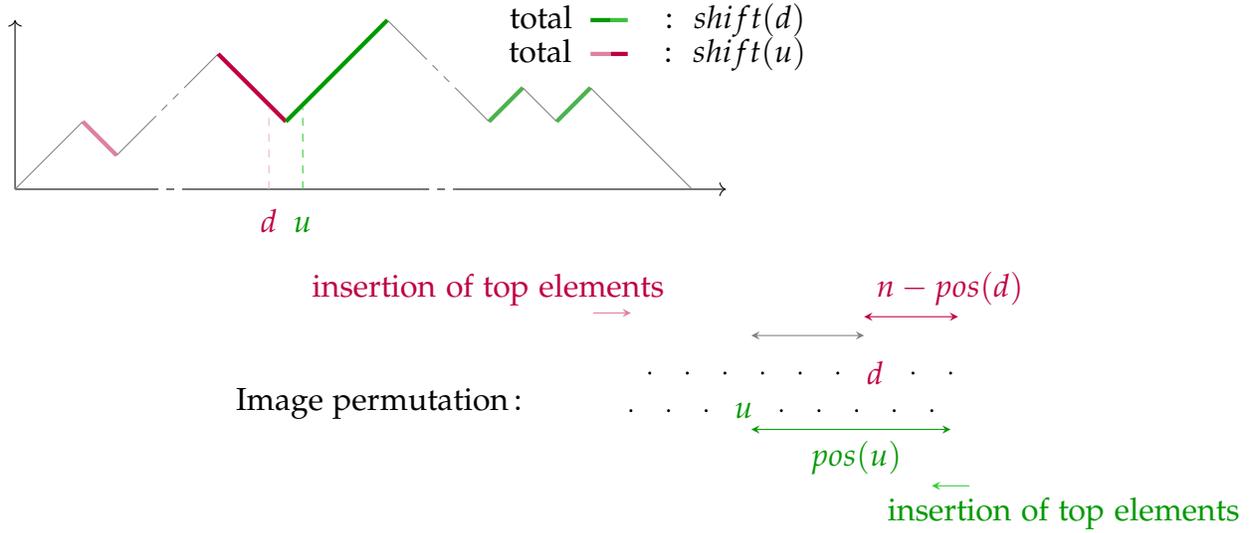
\begin{figure}
\begin{tikzpicture}[scale=0.45]
  \draw (0,0) -- (4,0);
  \draw[dashed] (4,0) -- (5,0);
  \draw (5,0) --(12,0);
  \draw[dashed] (12,0) -- (13,0);
  \draw[->] (13,0) --(21, 0);
  \draw[->] (0,0) --(0, 5);

  \draw (19, 4) node {total~~~~~~~~~ : ~$shift(u)$};
  \draw (19, 5) node {total~~~~~~~~~ : ~$shift(d)$};
  \draw[purple!50, ultra thick] (17,4) -- (17.6,4);
  \draw[purple, ultra thick] (17.6,4) -- (18.1,4);
  \draw[green!60!black, ultra thick] (17,5) -- (17.6,5);
  \draw[green!50!gray, ultra thick] (17.6,5) -- (18.1,5); 

  \draw[gray] (0,0) -- (2,2);  
  \draw[purple!50, ultra thick] (2,2) -- (3,1);  
  \draw[gray] (3,1) -- (4,2);  
  \draw[gray, dashed] (4,2) -- (5,3);  
  \draw[gray] (5,3) -- (6,4);  
  \draw[purple, ultra thick] (6,4) -- (8,2);  
  \draw[green!60!black, ultra thick] (8,2) -- (11,5);  
  \draw[gray] (11,5) -- (12,4);  
  \draw[gray, dashed] (12,4) -- (13,3);  
  \draw[gray] (13,3) -- (14,2);  
  \draw[green!40!gray, ultra thick] (14,2) -- (15,3);  
  \draw[gray] (15,3) -- (16,2);  
  \draw[green!40!gray, ultra thick] (16,2) -- (17,3);  
  \draw[gray] (17,3) -- (20,0);  

  \draw[purple!30, dashed] (7.5,0) -- (7.5,2.4);
  \draw[green!60!gray, dashed] (8.5,0) -- (8.5,2.4);
  \draw (7.5, -1.6) node[above] {\color{purple} $d$};
  \draw (8.5, -1.6) node[above] {\color{green!60!black}$u$};
  \end{tikzpicture}~~~~~~~~\vspace{.2cm}

  \hfill \begin{tikzpicture}[scale=0.5]
  \draw (-6.7, 0.2) node {Image permutation\,:};
  \draw[purple!50, ->, >=stealth] (-1,2.6) -- (0,2.6);
  \draw[green!60!gray, ->, >=stealth] (9,-2) -- (8,-2);
  \draw (-3.8, 2.5) node[above] {\color{purple} insertion of top elements};
  \draw (11.5, -2) node[below] {\color{green!60!black} insertion of top elements};
  \draw[purple, <->, >=stealth] (8.7,2.5) -- (6.2,2.5);
  \draw (8.5,2.5) node[above] {\color{purple} $n-pos(d)$};
  \draw[green!60!black, <->, >=stealth] (3.2,-0.5) -- (8.5,-.5);
  \draw (6, -0.5) node[below] {\color{green!60!black}$pos(u)$};
  \draw[gray, <->, >=stealth] (3.2,2) -- (6.2,2);
  \draw (0.5,1) node[scale=.1,circle, fill=black]{} ;  
  \draw (1.5,1) node[scale=.1,circle, fill=black]{} ;  
  \draw (2.5,1) node[scale=.1,circle, fill=black]{} ;  
  \draw (3.5,1) node[scale=.1,circle, fill=black]{} ;  
  \draw (4.5,1) node[scale=.1,circle, fill=black]{} ;  
  \draw (5.5,1) node[scale=.1,circle, fill=black]{} ;  
  \draw (6.5,1) node{\color{purple}$d$} ;  
  \draw (7.5,1) node[scale=.1,circle, fill=black]{} ;  
  \draw (8.5,1) node[scale=.1,circle, fill=black]{} ;  
  \draw (0,0) node[scale=.1,circle, fill=black]{} ;  
  \draw (1,0) node[scale=.1,circle, fill=black]{} ;  
  \draw (2,0) node[scale=.1,circle, fill=black]{} ;  
  \draw (3,0) node{\color{green!60!black}$u$} ;  
  \draw (4,0) node[scale=.1,circle, fill=black]{} ;  
  \draw (5,0) node[scale=.1,circle, fill=black]{} ;  
  \draw (6,0) node[scale=.1,circle, fill=black]{} ;  
  \draw (7,0) node[scale=.1,circle, fill=black]{} ;  
  \draw (8,0) node[scale=.1,circle, fill=black]{} ;  
\end{tikzpicture}
\caption{No smaller top element on the left of a given bottom element.}
\label{fig.small_peak_right}
\end{figure}

Finally, let us prove the last criterion is satisfied, that is to say that the
configuration of 1234 pattern on the following figure cannot occur.

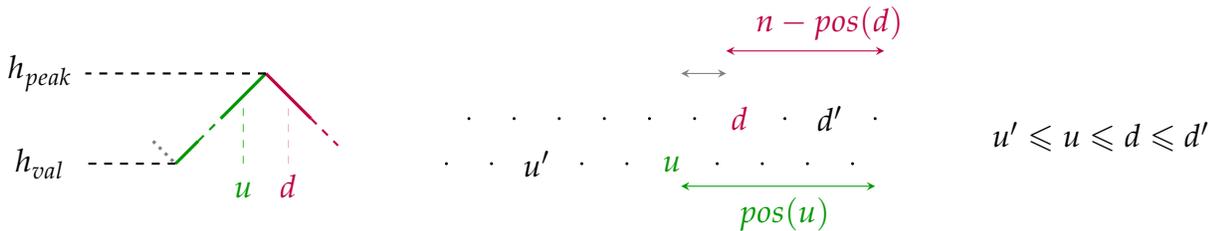
\hspace{-1cm}\begin{tikzpicture}[scale=0.6]
  \draw[gray, very thick, dotted] (-6.5,0.5) -- (-6,0);  
  \draw[green!60!black, very thick] (-6,0) -- (-5.5,0.5);  
  \draw[green!60!black, thick, dashed] (-5.5,0.5) -- (-5,1);  
  \draw[green!60!black, very thick] (-5,1) -- (-4,2);  
  \draw[purple, very thick] (-4,2) -- (-3,1);  
  \draw[purple, thick, dashed] (-3,1) -- (-2.4,0.4);  
  \draw[thick, dashed] (-6,0) -- (-8,0);
  \draw[thick, dashed] (-4,2) -- (-8,2);
  \draw (-9, 0) node {$h_{val}$};
  \draw (-9, 2) node {$h_{peak}$};

  \draw[purple!30, dashed] (-3.5,0) -- (-3.5,1.4);
  \draw[green!60!gray, dashed] (-4.5,0) -- (-4.5,1.4);
  \draw (-3.5, -1) node[above] {\color{purple} $d$};
  \draw (-4.5, -1) node[above] {\color{green!60!black}$u$};

  \draw (14.5, 0.6) node {$u'\leqslant u\leqslant d\leqslant d'$};
  \draw[purple, <->, >=stealth] (9.7,2.5) -- (6.2,2.5);
  \draw (8.5,2.5) node[above] {\color{purple} $n-pos(d)$};
  \draw[green!60!black, <->, >=stealth] (5.2,-0.5) -- (9.5,-.5);
  \draw (7.5, -0.5) node[below] {\color{green!60!black}$pos(u)$};
  \draw[gray, <->, >=stealth] (5.2,2) -- (6.2,2);
  \draw (0.5,1) node[scale=.1,circle, fill=black]{} ;  
  \draw (1.5,1) node[scale=.1,circle, fill=black]{} ;  
  \draw (2.5,1) node[scale=.1,circle, fill=black]{} ;  
  \draw (3.5,1) node[scale=.1,circle, fill=black]{} ;  
  \draw (4.5,1) node[scale=.1,circle, fill=black]{} ;  
  \draw (5.5,1) node[scale=.1,circle, fill=black]{} ;  
  \draw (6.5,1) node{\color{purple}$d$} ;  
  \draw (7.5,1) node[scale=.1,circle, fill=black]{} ;  
  \draw (8.5,1) node{$d'$} ;  
  \draw (9.5,1) node[scale=.1,circle, fill=black]{} ;  
  \draw (0,0) node[scale=.1,circle, fill=black]{} ;  
  \draw (1,0) node[scale=.1,circle, fill=black]{} ;  
  \draw (2,0) node{$u'$} ;  
  \draw (3,0) node[scale=.1,circle, fill=black]{} ;  
  \draw (4,0) node[scale=.1,circle, fill=black]{} ;  
  \draw (5,0) node{\color{green!60!black}$u$} ;  
  \draw (6,0) node[scale=.1,circle, fill=black]{} ;  
  \draw (7,0) node[scale=.1,circle, fill=black]{} ;  
  \draw (8,0) node[scale=.1,circle, fill=black]{} ;  
  \draw (9,0) node[scale=.1,circle, fill=black]{} ;  
\end{tikzpicture}

\vspace{-.3cm}
Observe\iflongversion first\fi, by considering once again the transformation at
the beginning of this proof, that the elements in $Bot(\sigma)$ (resp. $Top(\sigma)$)
that have a smaller bottom (resp. larger top) element to their left 
(resp. right) in the permutation are exactly those who do not jump.
Consider thus a bottom element $u$ and a larger top element $d$, both of
which do not jump. 
\iflongversion
Since all top elements larger than $d$ are inserted to its right, we may
assume without loss of generality that $d$ is the position of a step
from the slope down following that of $u$.
We need to show that $u$ is put to the right
of $d$ in the image permutation, so we compute the same difference
as before, but this time we need it to be negative (not just
non-positive).
Note that, unlike in the previous case, at least one of
the insertion positions will receive a minus $1$ when computed from the weight.
Hence, with $lenSlope(u)$ the length of the slope of $u$ and $h_{val}$ the height
of this slope's valley in the Dyck path:
\begin{align*}
dist(u,d) &= pos(u) +pos(d)-n -1 ~~(-1)\\
    &= wd(u)+shift(u)+wd(d)+shift(d)-n -1 ~~(-1) \\
    &= wd(u)+wd(d)+shift(u)+n-(lenSlope(u)+L(u))-n -1 ~~(-1) \\
    &= wd(d)+wd(u)-h_{val}-lenSlope(u) -1 ~~(-1) \\
    &\leqslant h_{peak}-h_val-lenSlope(u) -1 ~~(-1) = -1 ~~(-1)~.
\end{align*}
\else
Use similar arguments and distance computation as before to show that
$d$ is inserted to the left of $u$.
\fi
\end{proof}

\begin{corollary}
Let $n$ be a positive integer.
The map $\beta_n$ is a bijection from $WD_{2n}$ to $A_{2n}(1234)$
and is compatible with the concatenation product and 
the Schützenberger involution. Furthermore, the set of positions of the
steps $(1,1)$ of the Dyck path is the set of bottom elements (valleys) in 
its image up-down permutation.
\end{corollary}

Now that this bijection has been uncovered, it would be interesting to study
how it relates to product-coproduct prographs. Indeed, these objects are
more visual, and their geometric nature may allow to derive new properties.
To name one, it seems possible to endow them with a poset structure that
embeds Tamari lattices, and we conjecture it is a lattice itself: this is
a work in progress with Nicolas Borie.

\vspace{-.4cm}
\acknowledgements{

\vspace{-.2cm}
This research relied to the open-source software \texttt{SageMath}~\cite{sage} and the
combinatorics features developed by the \texttt{Sage-Combinat}
community~\cite{Sage-Combinat}.}
I would like to thank Nicolas Borie for his conjecture and for interesting
exchanges; as
well as Jean-Christophe Novelli for suggesting I take a look and for useful
discussions during the redaction phase.

\vspace{-.4cm}
\printbibliography


\end{document}